\documentclass[11pt]{article}

\usepackage{amsmath,amssymb,amsthm}
\usepackage{geometry}
\usepackage{graphicx}
\usepackage{microtype}
\usepackage{mathtools}
\usepackage[hidelinks]{hyperref}

\newtheorem{theorem}{Theorem}
\newtheorem{lemma}[theorem]{Lemma}
\newtheorem{proposition}[theorem]{Proposition}

\theoremstyle{definition}
\newtheorem{definition}[theorem]{Definition}
\theoremstyle{remark}
\newtheorem{remark}[theorem]{Remark}

\newcommand{\Z}{\mathbb Z}
\newcommand{\PH}{P^{\mathrm H}}
\newcommand{\PV}{P^{\mathrm V}}
\newcommand{\boxt}{\boxtimes}
\newcommand{\closeadj}{\mathrel{\simeq}}

\title{A Recursive Construction Improving the Lower Bound on the
Shannon Capacity of $C_7$}
\author{Yu Gao}
\date{}

\begin{document}

\maketitle

\begin{abstract}
We give a recursive reformulation and extension of the independent set
of size $134753$ in $C_7^{10}$ constructed by Itty, Rosin,
Carstensen, and Reichman~\cite{itty2026}.  We prove a product lemma
that combines gadgets of different dimensions while preserving the
required independence conditions.  Starting from the size-$367$
independent set in $C_7^5$ of Polak and Schrijver~\cite{polak2019},
the construction gives an explicitly specified independent set in
$C_7^{200}$.  Consequently,
\[
\Theta(C_7)\geq
3.2587891539086910161967650155\ldots .
\]
An accompanying program verifies the finite assertions about the
five-dimensional base gadget and performs the exact integer
computations used in the recursion.
\end{abstract}

\section{Introduction}

For a finite simple graph $G=(V,E)$, let $G^{\boxt d}$ denote its
$d$-th strong power.  Its vertex set is $V^d$, and two distinct
vertices $x=(x_1,\ldots,x_d)$ and $y=(y_1,\ldots,y_d)$ are adjacent
if and only if, for every $i$, either $x_i=y_i$ or
$\{x_i,y_i\}\in E$.  Let $\alpha(G)$ denote the independence number of
$G$.  The Shannon capacity of $G$ is~\cite{shannon1956}
\[
\Theta(G)=\sup_{d\geq 1}\alpha(G^{\boxt d})^{1/d}.
\]
We identify the vertices of $C_7$ with $\Z_7$.  Thus two distinct
vectors $x,y\in\Z_7^d$ are adjacent in $C_7^{\boxt d}$ if and only if
their circular distance is at most one in every coordinate.  We write
$C_7^d$ for $C_7^{\boxt d}$ for brevity.  Polak and Schrijver
constructed an independent set of size $367$ in
$C_7^5$~\cite{polak2019}; it remains the largest currently known such
set~\cite{itty2026}.  Itty et al.~\cite{itty2026} used it to construct
an independent set of size $134753$ in $C_7^{10}$ and obtained
\[
\Theta(C_7)\geq 134753^{1/10}
=3.258020737293245\ldots .
\]

The purpose of this note is to isolate and propagate a recursive
invariant hidden in that ten-dimensional construction.  Informally,
the invariant requires that no point of an auxiliary independent set
be confusable with points in both of two transversals.
This condition is stated precisely in Definition~\ref{def:gadget}.
Our main result is the following.

\begin{theorem}\label{thm:main}
There is an independent set in $C_7^{200}$ of cardinality
\[
\resizebox{0.97\textwidth}{!}{$
M_{40}=
4085567963379990282748597333793399773399910167372462112610203714626938509773125898252337545387866277249.
$}
\]
In particular,
\[
\Theta(C_7)\geq M_{40}^{1/200}
=3.2587891539086910161967650155206769\ldots .
\]
\end{theorem}

The subscript $40$ records that the construction uses forty
five-dimensional base gadgets.  The proof of
Theorem~\ref{thm:main} is given in Sections~\ref{sec:gadgets}
through~\ref{sec:iteration}.

\section{Gadgets}\label{sec:gadgets}

For vertices $u,v$ of a graph $G$, write
\[
u\closeadj v
\quad\Longleftrightarrow\quad
u=v\ \text{or}\ \{u,v\}\in E(G).
\] We say $u$ and $v$ are confusable in this case.
For $S\subseteq V(G)$, put
\[
N(S)=\{x\in V(G):x\closeadj s
    \text{ for some }s\in S\}.
\]
Thus $N(S)$ is the closed neighborhood of $S$.

\begin{definition}[Private pair]
Let $I$ be an independent set in $G$.  A pair $(r,q)$ is a \textit{private pair} for
$I$ if $r\in I$, $q\notin I$, and
\[
N(\{q\})\cap I=\{r\}.
\]
The point $r$ is the \textit{center}, and $q$ is its
\textit{private neighbor}; both are called endpoints of the private
pair.
\end{definition}

\begin{definition}[Gadget]\label{def:gadget}
A gadget in $G$ consists of the following data.
\begin{enumerate}
  \item An independent set $I\subseteq V(G)$ of size $a$.  We call
  $I$ the \textit{code} of the gadget.
  \item Private pairs with pairwise disjoint endpoint sets,
  \[
  (r_i,q_i),\qquad 1\leq i\leq t,
  \]
  and hence with distinct centers.
  \item Two complementary transversals $\PH,\PV$ of the pair
  endpoints: for every $i$, one of $r_i,q_i$ belongs to $\PH$ and the
  other belongs to $\PV$.  Both $\PH$ and $\PV$ are independent.
  \item An independent auxiliary set $X$ of $G$, such that
  \[
  X\cap N(\PH)\cap N(\PV)=\emptyset.
  \]
\end{enumerate}
Write
\[
\begin{aligned}
O&=X\setminus\bigl(N(\PH)\cup N(\PV)\bigr),\\
H&=X\cap N(\PH),\\
V&=X\cap N(\PV),
\end{aligned}
\]
and set
\[
(s,o,h,v)=(|X|,|O|,|H|,|V|).
\]
Thus $X=O\mathbin{\dot\cup}H\mathbin{\dot\cup}V$ and
$s=o+h+v$.
We call $(a,t,s,o,h,v)$ the parameter tuple of the gadget.
\end{definition}

The condition in Definition~\ref{def:gadget} also implies that $X$ is
disjoint from all private-pair endpoints.  Indeed, the two endpoints
of each private pair are confusable and lie in different
transversals.  Each endpoint therefore belongs to both
$N(\PH)$ and $N(\PV)$.

Let
\[
R=\{r_1,\ldots,r_t\},
\qquad
B=I\setminus R,
\qquad
Q=\{q_1,\ldots, q_t\}.
\]
The private-pair condition and the independence of $I$ imply
\begin{equation}\label{eq:B-avoids-endpoints}
B\cap N(R\cup Q)=\emptyset.
\end{equation}

\section{The product lemma}\label{sec:product}

We now combine two not necessarily identical gadgets.  Subscripts
$1,2$ refer to gadgets in graphs $G_1,G_2$, respectively.
For vertices of the strong product, closed adjacency is
coordinatewise:
\begin{equation}\label{eq:coordinate-confusability}
(x_1,x_2)\closeadj(y_1,y_2)
\quad\Longleftrightarrow\quad
x_1\closeadj y_1\ \text{and}\ x_2\closeadj y_2.
\end{equation}

Define the set-valued map
\[
L_1:X_2\longrightarrow 2^{V(G_1)}
\]
by
\[
L_1(x)=
\begin{cases}
R_1,&x\in O_2,\\
\PH_1,&x\in H_2,\\
\PV_1,&x\in V_2.
\end{cases}
\]
Define the set-valued map
\[
K_2:X_1\longrightarrow 2^{V(G_2)}
\]
by
\[
K_2(y)=
\begin{cases}
R_2,&y\in O_1,\\
\PV_2,&y\in H_1,\\
\PH_2,&y\in V_1.
\end{cases}
\]

\begin{lemma}\label{lem:product}
Let $\mathcal G_1,\mathcal G_2$ be gadgets with parameters
\[
(a_i,t_i,s_i,o_i,h_i,v_i),\qquad i=1,2.
\]
Define
\begin{align}
A_H&=\{(p,x):x\in X_2,\ p\in L_1(x)\},\label{eq:AH}\\
A_V&=\{(y,q):y\in X_1,\ q\in K_2(y)\},\label{eq:AV}
\end{align}
and
\begin{equation}\label{eq:product-code}
I_{12}=(B_1\times B_2)\cup A_H\cup A_V.
\end{equation}
Then $I_{12}$ is an independent set in $G_1\boxt G_2$ and
\begin{equation}\label{eq:a-product}
a_{12}=|I_{12}|
=(a_1-t_1)(a_2-t_2)+t_1s_2+s_1t_2.
\end{equation}

Moreover, $I_{12}$ is the code of a gadget $\mathcal G_{12}$ with
parameter tuple
$(a_{12},t_{12},s_{12},o_{12},h_{12},v_{12})$, where
\begin{align}
t_{12}&=t_1o_2+o_1t_2,\label{eq:t-product}\\
s_{12}&=s_1s_2,\label{eq:s-product}\\
o_{12}&=o_1o_2+(h_1+v_1)(h_2+v_2),\label{eq:o-product}\\
h_{12}&=h_1o_2+o_1v_2,\label{eq:h-product}\\
v_{12}&=v_1o_2+o_1h_2.\label{eq:v-product}
\end{align}
\end{lemma}

\begin{proof}
We first prove that $I_{12}$ from \eqref{eq:product-code} is independent.
$B_1\times B_2$ is independent.  No point of it is confusable with a
point of either $A_H$ or $A_V$, by
\eqref{eq:B-avoids-endpoints} and
\eqref{eq:coordinate-confusability}.

For a fixed $x\in X_2$, the set $L_1(x)$ is one of
$R_1,\PH_1,\PV_1$ and is therefore independent.  If $x,x'\in X_2$
are distinct, then $x\not\closeadj x'$ because $X_2$ is independent.
Consequently $A_H$ is independent.  The same
argument applies to $A_V$.

It remains to compare a point $(p,x)\in A_H$ with a point
$(y,q)\in A_V$.  Suppose first that $x\in H_2$.  Then
$p\in\PH_1$.  If $p\closeadj y$, then
$y\in H_1=X_1\cap N(\PH_1)$.  Hence $q\in\PV_2$, whereas
$x\in H_2$ is not confusable with any point of $\PV_2$.  Thus
$x\not\closeadj q$.  The case $x\in V_2$ is symmetric.

Finally suppose that $x\in O_2$.  Then $p\in R_1$.  Every $r_i$ lies
in exactly one of $\PH_1,\PV_1$.  If $p\closeadj y$, then $y$ is in
$H_1$ or $V_1$.  The set $K_2(y)$ is either $\PV_2$ or $\PH_2$, while
$x\in O_2$ is confusable with no point in either transversal.  Again
$x\not\closeadj q$.  Therefore no point of $A_H$ is confusable with a
point of $A_V$.

The three parts of \eqref{eq:product-code} are disjoint.  $B_1\times B_2$ is
disjoint from $A_H$ because $L_1(x)\cap B_1=\emptyset$ for every
$x\in X_2$, and it is disjoint from $A_V$ because
$K_2(y)\cap B_2=\emptyset$ for every $y\in X_1$.  An equality between
a point of $A_H$ and a point of $A_V$ would force a point of some
$X_i$ to be a private-pair endpoint, which the observation following
Definition~\ref{def:gadget} excludes.  Since each of
$R_i,\PH_i,\PV_i$ has size $t_i$, it follows that
$|A_H|=t_1s_2$ and $|A_V|=s_1t_2$.  This proves both independence and
the cardinality formula \eqref{eq:a-product}.

We next construct $\mathcal G_{12}$.  We have proved that its code
$I_{12}$ is independent.

For every private pair $(r_i,q_i)$ of $\mathcal G_1$ and every
$x\in O_2$, include the following private pair of $\mathcal G_{12}$:
\[
\bigl((r_i,x),(q_i,x)\bigr).
\]
For every $y\in O_1$ and every private pair $(r_j,q_j)$ of
$\mathcal G_2$, include the following private pair of
$\mathcal G_{12}$:
\[
\bigl((y,r_j),(y,q_j)\bigr).
\]
We verify the first family; the second is symmetric.  The center
$(r_i,x)$ belongs to $A_H$, because $x\in O_2$ and
$L_1(x)=R_1$.  Suppose that $(q_i,x)$ is confusable with
$(p,x')\in A_H$.  Since $x,x'\in X_2$ and $X_2$ is independent,
\eqref{eq:coordinate-confusability} implies $x'=x$.  Hence
$p\in R_1\subseteq I_1$, and the privacy of $(r_i,q_i)$ for $I_1$
gives $p=r_i$.  Thus $(q_i,x)$ is confusable with no point of $A_H$
other than $(r_i,x)$.  It is confusable with no point of
$B_1\times B_2$ by \eqref{eq:B-avoids-endpoints}.

It remains to exclude points of $A_V$.  If $q_i\closeadj y$ for some
$y\in X_1$, then $y\in H_1$ when $q_i\in\PH_1$, and $y\in V_1$ when
$q_i\in\PV_1$.  In the first case $K_2(y)=\PV_2$, and in the second
$K_2(y)=\PH_2$.  Because $x\in O_2$, it is confusable with no point
of either transversal.  Therefore $(q_i,x)$ is confusable with no
point of $A_V$.  This proves that
$\bigl((r_i,x),(q_i,x)\bigr)$ is a private pair for $I_{12}$.

Within the first family, pairs indexed by distinct
$(i,x)\in\{1,\ldots,t_1\}\times O_2$ have disjoint endpoint sets.
Within the second family, pairs indexed by distinct
$(y,j)\in O_1\times\{1,\ldots,t_2\}$ have disjoint endpoint sets.  The two families are disjoint from each other because every
$O_i\subseteq X_i$ is disjoint from the private-pair endpoints of
$\mathcal G_i$.  This proves \eqref{eq:t-product}.

Define the transversals of $\mathcal G_{12}$ by
\begin{align}
\PH_{12}
&=(\PH_1\times O_2)\cup(O_1\times\PV_2),\label{eq:new-PH}\\
\PV_{12}
&=(\PV_1\times O_2)\cup(O_1\times\PH_2).\label{eq:new-PV}
\end{align}
They are complementary transversals of the private pairs of
$\mathcal G_{12}$.  Each is
independent: its two pieces are independent products, and a
cross-conflict is excluded because every point in $O_i$ is
confusable with no private-pair endpoint of $\mathcal G_i$.

Take $X_{12}=X_1\times X_2$.  It is independent and disjoint from the
private-pair endpoints of $\mathcal G_{12}$.  Since $X_i$ is
independent, a point of $X_i$ is confusable with some point of
$O_i\subseteq X_i$ if and only if it itself belongs to $O_i$.
Using \eqref{eq:new-PH} and \eqref{eq:new-PV}, we obtain the disjoint
unions
\begin{align*}
O_{12}
&=(O_1\times O_2)
  \mathbin{\dot\cup}
  \bigl((H_1\cup V_1)\times(H_2\cup V_2)\bigr),\\
H_{12}
&=(H_1\times O_2)
  \mathbin{\dot\cup}
  (O_1\times V_2),\\
V_{12}
&=(V_1\times O_2)
  \mathbin{\dot\cup}
  (O_1\times H_2).
\end{align*}
In particular, $H_{12}\cap V_{12}=\emptyset$.  Taking cardinalities
gives
\eqref{eq:s-product}--\eqref{eq:v-product}.
\end{proof}

\section{The five-dimensional base gadget}\label{sec:base}

Let $I_0\subseteq\Z_7^5$ be the explicit independent set of size $367$
used in~\cite{polak2019,itty2026}.  We take $I=I_0$ as the code of the
base gadget.  As in~\cite{itty2026}, define the following eight pairs:
\[
\begin{array}{c|c|c}
j&r_j&q_j\\ \hline
0&(1,3,4,4,6)&(2,3,5,4,6)\\
1&(3,4,0,3,5)&(2,4,6,3,5)\\
2&(5,3,1,3,4)&(5,3,2,3,5)\\
3&(4,4,6,1,6)&(5,4,6,0,6)\\
4&(6,0,6,4,5)&(6,1,6,5,5)\\
5&(0,3,5,6,5)&(6,3,5,0,5)\\
6&(6,4,3,4,0)&(6,4,2,4,6)\\
7&(6,4,5,3,2)&(6,5,5,3,1)
\end{array}
\]
with all coordinates taken modulo $7$.  Put
\[
J_0=\{0,5,6\},\qquad J_1=\{1,2,3,4,7\},
\]
and
\begin{align*}
\PH&=\{r_j:j\in J_0\}\cup\{q_j:j\in J_1\},\\
\PV&=\{q_j:j\in J_0\}\cup\{r_j:j\in J_1\}.
\end{align*}

For $w=(w_0,\ldots,w_4)\in\Z_7^5$, define
\[
T(w)=(2-w_1,w_3,w_0,2-w_2,w_4).
\]
The map $T$ is an automorphism of $C_7^5$.  Let
$X_0=T(I_0)$ and set
\[
X=\bigl(X_0\setminus\{(2,4,6,3,5)\}\bigr)
  \cup\{(1,5,6,3,5)\}.
\]

\begin{proposition}[Base gadget]\label{prop:base}
The data above form a gadget in $C_7^5$ with
\[
(a,t,s,o,h,v)=(367,8,367,321,26,20).
\]
\end{proposition}

\begin{proof}
Each assertion is a finite check.  The verification program described
in Section~\ref{sec:verification} checks the following directly from
the $367$ vectors of $I_0$:
\begin{enumerate}
  \item $I_0$ and $X$ are independent and each has size $367$;
  \item for every $j$,
  \[
  N(\{q_j\})\cap I_0=\{r_j\};
  \]
  \item $\PH$ and $\PV$ are independent;
  \item $X$ is disjoint from the sixteen pair endpoints;
  \item among the points of $X$, exactly $321$ are confusable with
  neither $\PH$ nor $\PV$, exactly $26$ are confusable with $\PH$
  only, exactly $20$ are confusable with $\PV$ only, and none is
  confusable with both.
\end{enumerate}
These are precisely the conditions in
Definition~\ref{def:gadget}.
\end{proof}

\begin{remark}[Reproduction of the published ten-dimensional set]
Applying Lemma~\ref{lem:product} to two copies of the base gadget gives
\[
(367-8)^2+2\cdot8\cdot367=134753.
\]
Under the notation of~\cite{itty2026}, this construction reproduces
the functions denoted there by $h_j$ and $v_j$.
\end{remark}

\section{Iteration and the 200-dimensional code}\label{sec:iteration}

We count dimensions in five-dimensional blocks.  A gadget built from
$k$ base blocks lies in $C_7^{5k}$.  For gadgets constructed from
$k$ copies of the base gadget, the values of $s,o,t$ depend only on
$k$; we denote them by $s_k,o_k,t_k$.  The values of $a,h,v$ may also
depend on the chosen product tree.

\begin{lemma}\label{lem:closed-forms}
Every gadget obtained from $k$ copies of the base gadget has
\begin{equation}\label{eq:closed-forms}
s_k=367^k,\qquad
o_k=\frac{367^k+275^k}{2},\qquad
t_k=\frac{2(367^k-275^k)}{23}.
\end{equation}
\end{lemma}

\begin{proof}
Let $\delta=2o-s$.  Equations
\eqref{eq:s-product} and \eqref{eq:o-product} give
\[
s_{12}=s_1s_2,\qquad \delta_{12}=\delta_1\delta_2.
\]
At the base level, $(s,\delta)=(367,275)$.  Hence
$s_k=367^k$, $\delta_k=275^k$, and the formula for $o_k$ follows.

The formula for $t_k$ holds for $k=1$.  If it holds for $i$ and $j$,
then \eqref{eq:t-product} and
$o_\ell=(367^\ell+275^\ell)/2$ give
\[
\begin{split}
t_{i+j}
&=t_io_j+o_it_j\\
&=\frac{2(367^{i+j}-275^{i+j})}{23}.
\end{split}
\]
This proves \eqref{eq:closed-forms}.
\end{proof}

Let $M_1=367$.  For $k\geq2$, define
\begin{equation}\label{eq:DP}
M_k=\max_{\substack{i+j=k\\i,j\geq1}}
\left[
(M_i-t_i)(M_j-t_j)+t_is_j+s_it_j
\right],
\end{equation}
where $s_\ell,t_\ell$ are given by
\eqref{eq:closed-forms}.

\begin{proposition}\label{prop:Mk}
For every $k\geq1$, there is a gadget in $C_7^{5k}$ whose code
has size $M_k$ and whose values of $s,o,t$ are
$s_k,o_k,t_k$ as given by \eqref{eq:closed-forms}.
\end{proposition}

\begin{proof}
Induct on $k$.  The case $k=1$ is
Proposition~\ref{prop:base}.  For $k\geq2$, choose a split $i+j=k$
attaining the finite maximum in \eqref{eq:DP}, and apply
Lemma~\ref{lem:product} to the inductively constructed gadgets of
sizes $i$ and $j$.
\end{proof}

The following split tree produces $M_{40}$:
\[
1+1=2,\quad
1+2=3,\quad
2+3=5,\quad
5+5=10,\quad
10+10=20,\quad
20+20=40.
\]
For every $2\leq k\leq40$, the accompanying verification program
enumerates every split
\[
k=i+(k-i),\qquad 1\leq i\leq\left\lfloor\frac{k}{2}\right\rfloor,
\]
and evaluates \eqref{eq:DP} using exact integer arithmetic.  It
verifies that the displayed split tree attains the resulting maximum
at each of its nodes.  This exact computation gives
\[
\resizebox{0.97\textwidth}{!}{$
M_{40}=
4085567963379990282748597333793399773399910167372462112610203714626938509773125898252337545387866277249.
$}
\]
Proposition~\ref{prop:Mk} therefore proves the first assertion of
Theorem~\ref{thm:main}.  The lower bound on $\Theta(C_7)$ follows
immediately from its definition.

\section{Continuing balanced products indefinitely}

It is natural to ask whether repeatedly combining two identical
gadgets gives a still better asymptotic bound.  To keep this particular
balanced construction distinct from the dynamic-programming optimum
$M_k$, let $\widehat M_\ell$ denote the code size obtained after
$\ell$ balanced product steps.  Thus $\widehat M_0=367$, and the
number of five-dimensional base blocks at level $\ell$ is
$k_\ell=2^\ell$.  Put
\[
U_\ell=\frac{\widehat M_\ell}{367^{k_\ell}},
\qquad
T_\ell=\frac{t_{k_\ell}}{367^{k_\ell}}.
\]
By Lemma~\ref{lem:product} and \eqref{eq:closed-forms},
\[
T_\ell=
\frac{2}{23}
\left(1-\left(\frac{275}{367}\right)^{2^\ell}\right)
\]
and
\begin{equation}\label{eq:balanced-recurrence}
U_{\ell+1}=(U_\ell-T_\ell)^2+2T_\ell,
\qquad U_0=1.
\end{equation}

The resulting sequence of balanced root bounds has a limit.  Indeed,
set
\[
F_\ell=
\left(1-\frac{T_\ell}{U_\ell}\right)^2
+\frac{2T_\ell}{U_\ell^2},
\]
so that $U_{\ell+1}=U_\ell^2F_\ell$.  Since
$0\leq T_\ell\leq2/23$, recurrence
\eqref{eq:balanced-recurrence} shows inductively that $U_\ell\geq1$.
Consequently,
\[
\left(1-\frac{2}{23}\right)^2
\leq F_\ell\leq1+\frac{4}{23}.
\]
The sequence
\[
\frac{\log U_\ell}{2^\ell}
\]
therefore converges, because its successive increments are
$2^{-\ell-1}\log F_\ell$ and the latter form an absolutely convergent
series.  It follows that the balanced root bounds also converge.
Evaluation of \eqref{eq:balanced-recurrence} gives
\[
\lim_{\ell\to\infty}
\widehat M_\ell^{1/(5\cdot2^\ell)}
\approx
3.2586163193818009407725377551.
\]
This is smaller than the bound in Theorem~\ref{thm:main}.  Moreover,
there is no gain from taking the $200$-dimensional gadget of
Theorem~\ref{thm:main} and then repeatedly combining it with an
identical copy.  For $m\geq0$, let $\widetilde M_m$ denote the code
size after $m$ such products, and define
\[
\widetilde k_m=40\cdot2^m,\qquad
\widetilde U_m=\frac{\widetilde M_m}{367^{\widetilde k_m}},
\qquad
\widetilde T_m=\frac{t_{\widetilde k_m}}
                     {367^{\widetilde k_m}},
\]
where $\widetilde M_0=M_{40}$.  The normalized code sizes satisfy
\[
\widetilde U_{m+1}
=\bigl(\widetilde U_m-\widetilde T_m\bigr)^2
  2\widetilde T_m.
\]
Define the corresponding root bound by
\[
\widetilde\rho_m
=\widetilde M_m^{1/(5\widetilde k_m)}
=367^{1/5}\widetilde U_m^{1/(5\widetilde k_m)},
\]
and put $\tau=2/23$.  Exact integer arithmetic gives the following
strict inequality:
\[
\widetilde U_0
=1.058302615128940076\ldots>1+\frac{\tau}{2}.
\]
For every $m\geq0$, the value of $\widetilde T_m$ lies between
$\tau/2$ and $\tau$, since $(275/367)^{40}<1/2$.  If
$\widetilde U_m\geq1+\tau/2$, then the recurrence preserves this
inequality: its right-hand side is increasing in $\widetilde U_m$,
and at $\widetilde U_m=1+\tau/2$ its minimum over
$\widetilde T_m\in[\tau/2,\tau]$ is $1+\tau$.  Moreover,
\[
\widetilde U_{m+1}-\widetilde U_m^2
=\widetilde T_m
 \bigl(2-2\widetilde U_m+\widetilde T_m\bigr)
\leq0.
\]
It follows that
$\widetilde\rho_{m+1}\leq\widetilde\rho_m$ for every $m\geq0$.
For $m=1$, the construction has $80$ base blocks, dimension $400$,
and root bound
\[
3.258770378400339734\ldots
<3.258789153908691016\ldots .
\]
Thus neither the balanced construction from the base gadget nor
continued balanced products starting from the $200$-dimensional
gadget improve Theorem~\ref{thm:main}.

\section{Verification}\label{sec:verification}

The source code and input data are available on
GitHub~\cite{verification2026}.  To make the computation reproducible,
we refer to the fixed commit
\texttt{b13031ba76e3}.  With Python~3.11 or later, the verification is
run from the repository root by
\begin{verbatim}
python verify_from_c5_phase_recursion.py
\end{verbatim}
The program reads \texttt{inputs/R367.txt}, checks
Proposition~\ref{prop:base}, and evaluates the recursion using exact
integer arithmetic.  In particular, for every $2\leq k\leq40$, it
enumerates all splits in \eqref{eq:DP} and reproduces the integer and
the lower bound in Theorem~\ref{thm:main}.  The finite claims of the
theorem depend only on these exact checks, not on the exploratory
floating-point scan also provided by the program.

\section*{Statement on the use of large language models}

During the development of this work, ChatGPT 5.6 Sol implemented all
the code and expanded the proofs.

\end{document}